\documentclass[12pt]{amsart}

\usepackage[pdftex]{graphicx}

\usepackage[T2A]{fontenc}
\usepackage[utf8]{inputenc}
\usepackage[english]{babel}

\usepackage{amsmath}
\usepackage{amssymb}
\usepackage{amscd}
\usepackage{amsfonts,amsmath,amsthm}

\usepackage[matrix,arrow,curve]{xy}
\usepackage{xfrac}
\usepackage{color}

\newtheorem{theorem}{Theorem}
\newtheorem{prop}{Proposition}
\newtheorem{cor}{Corollary}
\newtheorem{lemma}{Lemma}
\newtheorem*{main lemma}{The main lemma}

\newtheorem{dfn}{Definition}

\title[Each simple $n$-polytope   has a point with  $2n+4$ normals ]{Each simple convex polytope from $\mathbb{R}^n$  has a point with  $2n+4$ normals to the boundary }

\author{Ivan Nasonov, Gaiane Panina}

\address{I. Nasonov: Saint Petersburg University, 7/9 Universitetskaya nab., St. Petersburg, 199034 Russia   wanua-nasonov-i04@yandex.ru;  G. Panina: St. Petersburg department of Steklov institute of mathematics  RAS, Fontanka 27, St. Petersburg, 191023 Russia
 gaiane-panina@rambler.ru  }
\date{\today}

\begin{document}

\begin{abstract}

We prove that for $n>3$ each generic simple polytope in $\mathbb{R}^n$   contains a point  with at least $2n+4$ emanating normals to the boundary. 

This result is a piecewise-linear counterpart of a long-standing problem about normals to smooth convex  bodies.
\end{abstract}

\maketitle

\section{Introduction}

\medskip

It is conjectured since long that for any smooth convex body $\mathbf{P}\subset \mathbb{R}^n$ there exists a point in its interior   which belongs to at least $2n$ normals from different points on the boundary of $\mathbf{P}$. The proof of the conjecture is a simple exercise for $n=2$. There are two known proofs for  $n=3$:  a geometrical one by E. Heil \cite{Heil},  and a topological one by J. Pardon~\cite{Pardon}. 
 For   $n=4$ the conjecture was proved by J. Pardon \cite{Pardon}, and (to the best of our knowledge) nothing is known for higher dimensions.
In this paper we solve a similar problem for   simple convex polytopes. 

\medskip

 Let $\mathbf{P}\subset \mathbb{R}^n$  be a compact convex polytope with non-empty interior and let $y\in Int~\mathbf{P}$.
  A \textit{normal} to the boundary of $\mathbf{P}$ emanating from $y$ is a line $yz$ such that $z\in \partial \mathbf{P}$, and  $yz$ is orthogonal to some supporting hyperplane containing $z$.   The point $z$ is called the\textit{ base} of this normal.

\medskip

Our main result is:

\begin{theorem}\label{Thmmain}
  For $n>2$,  each generic  simple 
    polytope  $\mathbf{P} \in \mathbb{R}^n$   has a point with (at least) $2n+4$ emanating normals to the boundary.
\end{theorem}
 
 Important remarks are:
 \begin{itemize}
   \item Each polytope has a point with $2n+2$ normals. This follows  directly from the Morse count, see \cite{NasPanSiersma} for details.
   \item The statement of the theorem is not true for $n=2$: a triangle has no point with eight normals.
   \item For $n=3$ the bound is exact: there exists a tetrahedron with at most $10$ normals emanating from its inner points.
   \item  There is already a proof  for $n=3$ for all generic polytopes, not only simple ones (see \cite{NasPan},  and  a slightly different proof in \cite{Angst}) which is quite involved.
    \item The proof for $n=4$ is also involved, but that for $n>4$ is simpler.
   For this reason we conjecture that for large $n$ there is a point with $2n+6$ normals to the boundary of $\mathbf{P}$.
 \end{itemize}

  \medskip

 The \textbf{main idea of the proof in a nutshell}   looks like this:
  
  (1) In the smooth case the bases of the normals  emanating from $y$ are  the critical points of the \textit{squared distance function } $$SQD^\mathbf{P}_y(x)=|x-y|^2:\partial \mathbf{P}\rightarrow \mathbb{R}.$$ Since $SQD^{\mathbf{P}}_y$ is generically a  Morse function, Morse theory is applicable.
  The Morse-theoretic machinery extends also to polytopes \cite{NasPanSiersma}, so counting normals to the boundary is equivalent to counting critical points of $SQD^\mathbf{P}_y$.

  (2)  We  choose a very special point $y\in\mathbf{P}$, force it to travel along some carefully chosen trajectories and watch the bifurcations of critical points of $SQD^\mathbf{P}_y$. This allows  us to prove that if a polytope  $\mathbf{P}$ has no point with $2n+4$ normals, then all its $(n-3)$-faces are \textit{skew}, that is, the adjacent cones have a very specific shape.

 (3) For $n=4$ we need  to once again look at bifurcations, in the same spirit of the item (2), but using other trajectories. 
 
 (4) For $n>4$ we arrive at some non-realizable combinatorics.

 \medskip
 
 The paper relies very much on ideas developed in \cite{NasPan} and \cite{NasPanSiersma}. So we refer the reader to these papers for more detailed explanations of the backgrounds and also for more figures.

\medskip


\section{Definitions and preliminaries}\label{SecDefPrelim}

Let $\mathbf{P}\subset \mathbb{R}^n$ be a compact convex polytope with non-empty interior.
\textit{A face }of $\mathbf{P}$ is the intersection of $\mathbf{P}$ with a support hyperplane.
The faces of dimension $k$ will be called\textit{ $k$-faces} for short. $0$-faces and $1$-faces are called vertices and edges, as usual.
The faces of dimension $n-1$ are called \textit{facets}.

 Throughout the paper we assume that $\mathbf{P}$  is
\textit{simple}, that is, each vertex has exactly $n$ incident edges.

 The \textit{cone }of a face $F$ (denoted by $CF$) is defined as the positive span of $\mathbf{P}$  assuming that the origin of the ambient space $\mathbb{R}^n$ lies in the relative interior of $F$. In plain words, the cone  $CF$ is bounded by extensions of the facets containing $F$.

Let $F$ be a face of  $\mathbf{P}$. The \textit{active region} $\mathcal{AR}(F)$  \cite{NasPanSiersma} is the set of all points $y\in Int\, \mathbf{P}$ such that $F$ contains the base of some normal emanating from $y$. 

The\textit{ bifurcation set} $\mathcal{B}(\mathbf{P})$ \cite{NasPanSiersma} is defined as the union of the boundaries of active regions. It is the piecewise linear counterpart of the focal set that appears in the smooth case.

The{ bifurcation set}  cuts $\mathbf{P}$  into a number of (open) cameras. For points in one and the same camera, the number of emanating normals is one and the same, 
and the bases of the normals lie in one and the same faces.

\medskip

Although the function $SQD^\mathbf{P}_y$ is  non-smooth,
it can be treated as a Morse function.  So counting critical points  of $SQD^\mathbf{P}_{y}$ amounts to counting normals, exactly as it is in the smooth  case.

\begin{lemma} \cite{NasPanSiersma}

For $y \notin \mathcal{B}(\mathbf{P})$, we have:
\begin{enumerate}
  \item Local maxima of $SQD^\mathbf{P}_y$  are attained at some  vertices of $\mathbf{P}$.
  \item Local minima of $SQD^\mathbf{P}_y$ are attained at some facets of $\mathbf{P}$.
  \item Critical points of Morse index $m$ are attained at some  $n-m-1$-faces of $\mathbf{P}$.
   \item If $y$ crosses transversally one sheet of $\mathcal{B}(\mathbf{P})$, a pair of critical points of  $SQD^\mathbf{P}_{y}$ either dies or is born.
   The Morse indices of this pair of points differ by one.
   \item The number of critical points is even.
\end{enumerate}

\end{lemma}

\textbf{Convention:}  From now on, we call local maxima (local minima) just maxima and minima for short.

\medskip

We say that a point $y\in Int(\mathbf{P})$ \textit{projects} onto a face $F$  if the orthogonal projection of $y$ to the affine hull of $F$ lies in the interior of the face $F$.

\bigskip

The function $SQD^\mathbf{P}_y$ has a critical point at a face $F$ iff two conditions hold:
\begin{enumerate}
  \item $y$ projects to $F$.
  \item $\overrightarrow{zy}$  makes angles  non-greater than $\pi/2$ with $zx$, where $z$ is the projection of $y$, and $x$ ranges over $\mathbf{{P}}$.
\end{enumerate}

\newpage

\textbf{Genericity convention} \label{DefGen} \cite{NasPan}: A polytope $\mathbf{P}$ with the vertex set $\mathbf{V}$ is  \textit{generic} if affine hulls of any two subsets of
$\mathbf{V}$ are neither parallel nor orthogonal, unless this is dictated by unavoidable reasons (for example, an edge is always parallel to a face containing the edge).

\medskip

Important observations are: (1) {For a generic polytope, the sheets of $\mathcal{B}(\mathbf{P})$ intersect transversally.}
(2) {Any polytope can be turned to a generic one by a close to $id$ projective transform.}

\section{Some necessary tools from spherical geometry}

A \textit{spherical polygon }  is a polygon with geodesic edges lying in the standard sphere $S^2$. It is always supposed to fit in an open hemisphere.
A \textit{spherical polytope} in   the standard sphere $S^k$ is defined analogously.

A vertex $V$ of a polytope $\mathbf{P}\subset \mathbb{R}^n$ yields a convex spherical polytope $Spher(V)\subset S^{n-1}$  which is the intersection of a small sphere centered at $V$  with  the polytope $\mathbf{P}$.  Although the sphere is small, we assume that it is equipped with the metric of the unit sphere.

Analogously, an $m$-face $F$ of a polytope $\mathbf{P}$ yields a convex spherical polytope $Spher(F)\subset S^{n-m-1}$  which is the intersection of a small sphere centered at any point $x$ in the relative interior of $F$   with  the polytope $\mathbf{P}$  and with  the $(n-m)$-plane passing through $x$ and orthogonal to $F$.  

\medskip

We borrow the definition and properties of \textit{ skew spherical triangles} from \cite{NasPanSiersma}. 

\begin{dfn}\cite{NasPanSiersma}
  A spherical triangle $V_1V_2V_3$ is nice if there is a point $X$
in its interior such that
\begin{enumerate}
  \item $X$ projects to all the three edges of the triangle, and
  \item $|XV_i|< \pi/2 \ \ \forall i$. 
\end{enumerate}

Otherwise a spherical triangle is called skew.
\end{dfn}

\begin{lemma}\label{LemmaSkewTriang}\cite{NasPanSiersma}

For a skew triangle, up to renumbering, 
\begin{enumerate}
  \item the angles  at $V_1$ and $V_3$ are acute, the angle at $V_2$ is obtuse,
  \item the edges $V_1V_2$ and $V_1V_3$  are longer than $\pi/2$, the edge $V_2V_3$ is shorter than $\pi/2$,
  \item if $Z\in V_1V_2V_3$ and $|V_1Z|<\pi/2, |V_3Z|<\pi/2$ then $|V_2Z|<\pi/2$.
\end{enumerate}

\end{lemma}

For a spherical convex polytope $P$ and a point $Y\in P$ denote by $SQD_Y^P:\partial P \rightarrow \mathbb{R}$  the spherical analog of $SQD^\mathbf{P}$.
Spherical squared  distance functions inherits  the properties  of the Euclidean squared distance function  only if the distance from $Y$ to the base of the corresponding normal are smaller than $\pi/2$.

\section{Skew and nice faces}

 For an $(n-3)$-face $F$, $Spher(F)$ is a spherical triangle.
\begin{dfn}
An  $(n-3)$-face $F$ of $\mathbf{P}$ is called \textit{nice (resp., skew)} if $Spher(F)$ is nice (resp., skew).
\end{dfn}

Let us extend this notion  to faces of other dimensions.
The motivation comes from counting normals to the faces of a cone: 
    for $y\in CF$ let $SQD^{CF}_y: \partial CF \rightarrow \mathbb{R}$ be the square distance function.
If  $F$ is a nice $(n-3)$-face of $\mathbf{P}$, then there exists a point $y\in CF$ such that $SQD_y^{CF}$ has at least $7$ critical points. 
By genericity we may assume that the set of points with at least seven normals  has a non-empty interior.

\begin{dfn}
  For $k=3,4,\ldots$,  a $(n-k)$-face $F$ of $\mathbf{P}$ is called \textit{nice} if there exists a point $y\in CF$, such that $SQD_y^{CF}$ has at least $2k+1$ critical points on the faces of $CF$.
\end{dfn}

The first observations are following:

\begin{lemma}\label{lemmaGen}
  In the above notation,  \begin{enumerate}
        \item The number of critical points of $SQD^{CF}_y$ is  odd.
\item  If an $(n-k)$-face $F$ is nice, there is a point $y\in \mathbf{P}$  with at least  $2k+1$ normals to the faces of $\mathbf{P}$  containing $F$
        \item Let $F, F'$ be some faces such that $F'$ is a facet of $F$. 
        Then  the cone $F'$ equals $CF$ intersected with a halfspace $h^+$ bounded by some hyperplane $h$.
        
    \end{enumerate}
\end{lemma}
\begin{proof}
  (1) This follows from ''attaching handles'' Morse count \cite{Milnor}:  on the one hand, the preimage $\Big(SQD_y^{{C}F}\Big)^{-1 }[0,r]$ for a fixed $y$ and big $r$ is a ball, so its Euler characteristic is $1$. On the other hand,  
   the preimage decomposes into handles of different dimensions whose Euler characteristics sum up to $1$. Each handle corresponds to a critical point, so the number of critical points is odd. 
  
  (3) For this purpose choose $y$ lying close to $F$. This ensures that the bases of the normals to $CF$ lie on the faces of $\mathbf{P}$.

  (3) Follows from simplicity of $\mathbf{P}$.
\end{proof}

\begin{prop}\label{propKey}
Each face of a nice face of $\mathbf{P}$ is nice.

\end{prop}

\begin{proof} We shall prove that  if  $F'\subset F$  are some faces, $\dim F = n-k,\  \dim F' = n-k-1$, and $F$ is nice, then $F'$ is also nice.
Let $h$ be a hyperplane given by Lemma \ref{lemmaGen}. Pick a point $y\in CF$ for which $SQD^{CF}_y$ has at least $2k+1$ critical points. Denote these points by $x_1, x_2, \ldots\in \partial CF$.

We can assume that $y \in CF'$,  and  $x_1,x_2,\ldots\in \partial CF'$.
Indeed, otherwise one translates $y$ by a vector  parallel to  $F$.

\textbf{Case 1.} 

 If $SQD_y^{CF}$ already has at least $2k+3$ critical points,  we are done since
 the points $x_1, x_2,\ldots,x_{2k+3}$  are the critical points of $SQD_y^{CF'}$.

\textbf{Case 2.} Otherwise, the number of critical points $SQD_y^{CF}$ is exactly $2k+1$. Now $y\in CF'$ and $x_1,\ldots,x_{2k+1}\in \partial CF \cap h^+\subset \partial CF'$ are all the critical point of $SQD_y^{CF'}$. 

Let the point $y$  travel by a  straight line parallel to $F$  such that the distance to the hyperplane $h$ decreases. Consider the\textit{ first event}, which is (by definition)  

 (a) either a bifurcation of $SQD_y^{CF'}$, or 
 
 (b) the point $y$ reaches  $h$.

Case (a).   By genericity convention we may assume that $y$ crosses transversally exactly one sheet of $\mathcal{B}(\mathbf{P})$ at a time. During the travel, the pairwise distances between the points $x_1,\ldots,x_{2k+1}$ and $y$ persist. So, they never collide with each other. The only possible bifurcation is a bifurcation with some extra critical point, not from this list. 

Case (b). Just before reaching $h$, $SQD_y^{CF'}$ has additional local minima on $h$, which was not on the list.

Anyway we have at least $2k+2$ critical points. By Lemma \ref{lemmaGen}  there are  at least $2k+3$.
\end{proof}

\begin{cor}\label{corGen}
If $\mathbf{P}$ has a nice $(n-3)$-face $F$, then there exists $V\in Vert(\mathbf{P})$ and a point $y\in \mathbf{P}$, such that $SQD^{\mathbf{P}}_y$ has at least $2n+1$ critical points on the faces of $\mathbf{P}$, containing $V$.
\end{cor}

\begin{proof}
    Follows from Lemma \ref{lemmaGen} and Proposition \ref{propKey}.
\end{proof}

\begin{prop}\label{PropNo-niceGen}
If $\mathbf{P}$ has at least one nice face, then there is a point in its interior with at least $2n+4$ normals to the boundary.
\end{prop}

\begin{proof}
By Corollary \ref{corGen}, there is a vertex $V$ and a point $y\in \mathbf{P}$, such that $SQD^\mathbf{P}_y$ has at least $2n+1$ critical points on the faces of $\mathbf{P}$, containing $V$. We can assume that $y$ lies sufficiently close to $V$. Besides we have a global maximum, so altogether there are at least $2n+2$ critical points. If there are $2n+4$ critical points, we are done. Otherwise there are exactly $2n+2$ critical points since the number is always even. So, we don't have any critical points except those listed above. 

Let $y$ travel straight away from $V$ until the first event. If $y$ reaches $\partial \mathbf{P}$, we have extra minimum just before it. In case of bifurcation, old $2n+1$ critical points cannot collide and die, since distances between them increases. Also they cannot collide with the global maximum. So the first bifurcation is a birth of two new critical points. (Here again by genericity convention we  assume that $y$ crosses transversally exactly one sheet of $\mathcal{B}(\mathbf{P})$ at a time.)
\end{proof}

\section{Proof of Theorem \ref{Thmmain} for $n=4$}

Due to Proposition \ref{PropNo-niceGen} it remains to prove the theorem for the case  when all the edges of $\mathbf{P}$ are skew. 

Take any vertex $V$ of $\mathbf{P}$. The associated spherical polytope $Spher(V)\subset S^3$ is a convex spherical tetrahedron  $\Delta$. The vertices of $\Delta$ correspond to edges of $\mathbf{P}$ that emanate from $V$, so all the vertices of $\Delta$ are skew. 

\begin{lemma}\label{LemmaSpherTetr}
  The tetrahedron  $\Delta$ contains a point $Y$ in its interior such that 
   $$SQD^{\Delta}_Y:\partial \Delta \rightarrow \mathbb{R}$$ has at least $8$ critical points such that the distances to the critical points is smaller than $\pi/2$.
\end{lemma}
\begin{proof}
  Take $Y\in \Delta$ which minimizes the Hausdorff distance to $\partial \Delta$. Then automatically the distances from $Y$ to any other point of $\Delta$ is smaller than $\pi/2$.  In particular this means that minima, saddles and maxima of $SQD_Y^\Delta$ are attained respectively at facets, edges and vertices of $\Delta$ (which is not always true for spherical polytopes).
  
  Observe that the edges of $\Delta$  correspond to $2$-faces of $\mathbf{P}$.  Therefore, by Lemma \ref{LemmaSkewTriang}, the spherical simplex $\Delta$ has exactly four acute edges. These edges form a closed broken line $\mathcal{A}$.  Denote by $SQD_Y^{\mathcal{A}}: \mathcal{A}\rightarrow \mathbb{R}$  the spherical distance from $Y$ to the points on $\mathcal{A}$.

The following  fact is a spherical counterpart of Lemma 6 from \cite{NasPan}.
\begin{itemize}
  \item Each minimum of  $SQD_Y^{\mathcal{A}}$ is a  saddle of $SQD_Y^{\Delta}$.
  \item Each maximum of $SQD_Y^{\mathcal{A}}$ is a maximum of $SQD_Y^{\Delta}$.
\end{itemize}
  
  Let us comment on this fact a little. $SQD_Y^{\mathcal{A}}$ has a minimum at an edge $e$ iff $Y$ projects  to the edge, and distance to the projection is smaller than $\pi/2$. This is already sufficient for $SQD_Y^{\Delta}$ to have a saddle at $e$ since $e$ is acute. 
  
  Besides, $SQD^{\mathcal{A}}_Y$ never attains a minimum at a vertex of $\mathcal{A}$. Indeed, if $V_1V_2V_3$ is a skew triangle (as in Lemma \ref{LemmaSkewTriang}) then there is no point $Z \in V_1V_2V_3$ that $|V_1Z|>\pi/2$ and $|V_3Z|>\pi/2$.
  
  For the second statement, we need the following observation. Assume that $V_1V_2V_3$ is a skew triangle, and a point $X$ in it  is such that $|XV_1|$ and $|XV_3|$  are smaller than $\pi/2$. Then $|XV_2|<\pi/2$, see Lemma \ref{LemmaSkewTriang}.
  
  \medskip
  Denote by $d$ the Hausdorff distance from $Y$   to $\partial \Delta$ and consider three cases.
  
  \begin{enumerate}
    \item $d$ is attained at four vertices of $\Delta$.  Then $SQD_Y^{\Delta}$ has four maxima, therefore at least three saddles and one minimum.
    \item $d$ is attained at three vertices of $\Delta$. Then $SQD_Y^{\Delta}$ has three maxima, hence $SQD_Y^{\mathcal{A}}$ has three maxima, 
    hence $SQD_Y^{\mathcal{A}}$ has three minima, hence $SQD_Y^{\Delta}$ has three saddles. There is also at least one minimum, and since the number of critical points is even, we get $8$.
    \item $d$ is attained at two vertices of $\Delta$, say, at $U$ and $V$. Then $Y$ belongs to the edge $UV$.  Otherwise one can shift $Y$ and  make the Hausdorff distance smaller. But none of $U,V$ is a maximum for a point lying on $UV$  by the properties of skew triangles, see Lemma \ref{LemmaSkewTriang}.
  \end{enumerate}
\end{proof}

Now  prove the theorem.
Take any vertex $V$ of the polytope $\mathbf{P}$. By Lemma \ref{LemmaSpherTetr}, $V$ is a nice face. Indeed, 8 <<short>> critical points of $SQD^{\Delta}_Y$ correspond to 8 critical points of $SQD^{CV}_Y$. 

So to complete the theorem it remains to apply   Proposition \ref{PropNo-niceGen}.

\section{Proof of Theorem \ref{Thmmain} for $n>4$}

Assume the contrary, that a polytope $\mathbf{P}$ has no point with $2n+4$ normals.

 By Proposition \ref{PropNo-niceGen}, all the $(n-3)$-faces of $\mathbf{P}$ are skew. Hence by Lemma \ref{LemmaSkewTriang}  each $(n-3)$-face is contained in two acute $(n-2)$-faces and in one obtuse $(n-2)$-face. Paint acute $(n-2)$-faces  \textit{red} and the obtuse $(n-2)$-face \textit{blue}. 
 We arrive at a coloring of all $(n-2)$-faces of $\mathbf{P}$ with the key combinatorial property:
 \textit{each $(n-3)$-face has exactly two adjacent red $(n-2)$-faces and one blue $(n-2)$-face.}

 Prove that such coloring cannot exist  and thus get a contradiction.
Choose any vertex $V$ and cut it off from $\textbf{P}$ by a hyperplane $h$. Then $h\cap\mathbf{P}$ is a simplex $\Delta^{n-1}$, which inherits the coloring and the key property: the $(n-3)$-faces of $\Delta^{n-1}$ are painted red and blue  such that each of the $(n-4)$-face has exactly two adjacent red faces and one blue face.  Such a coloring does not exist. It is an easy fact; one of the ways to prove it is to further cut the vertices and to reduce the dimensions of simplices until we get $\Delta^{4}$. A simple case analysis shows that   $\Delta^{4}$ is not colorable in this sense.
 \qed

\bigskip

\textbf{Remark.
}
This proof breaks down for four-dimensional polytopes. Moreover, there exist colorable  polytopes in $\mathbb{R}^4$. For instance, a  cube is colorable.

\subsection*{Acknowledgement}
The work of I. Nasonov  was performed at the Saint Petersburg Leonhard Euler
International Mathematical Institute and supported by the Ministry of Science and Higher
Education of the Russian Federation (agreement no. 075–15–2025–343).


\end{document}